%% file: 4-paths.tex
\documentclass[preprint]{article}

\usepackage{setspace}
\usepackage{fullpage}
\usepackage{datetime}

\usepackage{latexsym} 

\usepackage[utf8]{inputenc}
\usepackage[english]{babel}

\usepackage{amsfonts,amsmath,amssymb,amsbsy,amsthm,enumerate}

\usepackage[mathscr]{euscript}
\usepackage{pstricks}
\usepackage{multirow}
\usepackage{verbatim}
\usepackage{color}
\usepackage{comment} 

\newtheorem{theorem}{Theorem}[section]

\newtheorem{conjecture}[theorem]{Conjecture}

\newtheorem{lemma}[theorem]{Lemma}
\newtheorem{corollary}[theorem]{Corollary}

\newtheorem{fact}[theorem]{Fact}

\newtheorem{definition}[theorem]{Definition}

\usepackage{tikz}
\usepackage{xifthen}	
\usepackage{subcaption} 
\usetikzlibrary{calc,positioning,decorations.pathmorphing,decorations.pathreplacing}
\usepackage{./figstyle}

\newcommand\D{\mathcal{D}}

\newcommand\B{\mathcal{B}}
\newcommand\prehang{\text{Hang}}

\usepackage{graphicx}

\title{Decomposing 8-regular graphs into paths of length 4}

\author{F. Botler 
\thanks{This research has been partially supported by CNPq
  Projects (Proc. 477203/2012-4 and {456792/2014-7}), Fapesp Project
  (Proc. 2013/03447-6). 
  F. Botler is supported by Fapesp (Proc. 2014/01460-8 and
  2011/08033-0) and CAPES (Proc. 1617829).}
\and A. Talon}

\date{{\small Instituto de Matem\'atica e Estat\'istica} \\
	{\small Universidade de S\~ao Paulo}\\
	{\small ENS Lyon}\\
	\vspace{.4cm}\today\vspace{-.3cm}}

\begin{document}

\maketitle

\begin{abstract}
	A \(T\)-decomposition of a graph \(G\) is a set of 
	edge-disjoint copies of \(T\) in \(G\) that cover the
	edge set of \(G\).
	Graham and H\"aggkvist (1989) conjectured that
	any \(2\ell\)-regular graph \(G\) admits a \(T\)-decomposition
	if \(T\) is a tree with \(\ell\) edges.
	Kouider and Lonc (1999) conjectured that, in the special case where \(T\)
	is the path with \(\ell\) edges,
	\(G\) admits a \(T\)-decomposition \(\D\)
	where every vertex of \(G\) is the end-vertex of exactly
	two paths of \(\D\),
	and proved that this statement holds when \(G\) has girth 
	at least \((\ell+3)/2\).
	In this paper we verify Kouider and Lonc's Conjecture
	for paths of length \(4\).
	
\end{abstract}

\let\thefootnote\relax\footnote{E-mails: fbotler@ime.usp.br (F. Botler), 
				alexandre.talon@ens-lyon.org (A. Talon).}

\section{Introduction}

A decomposition of a graph \(G\) is a set \(\D\) of edge-disjoint
subgraphs of \(G\) that cover the edge set of \(G\).
Given a graph \(H\), 
we say that \(\D\) is an \(H\)-decomposition of \(G\) 
if every element of \(\D\) is isomorphic to \(H\).
Ringel~\cite{Ringel64} conjectured that the complete graph \(K_{2\ell+1}\) 
admits a \(T\)-decomposition 
for any tree \(T\) with \(\ell\) edges.
Ringel's Conjecture is commonly confused with the \emph{Graceful Tree Conjecture}
that says that any tree \(T\) on \(n\) vertices admits 
a labeling \(f\colon V(T)\to \{0,\ldots,n-1\}\)
such that \(\{1,\ldots,n-1\}\subseteq \{|f(x)-f(y)|\colon xy\in E(T)\}\).
Since the Graceful Tree Conjecture implies Ringel's Conjecture~\cite{Rosa67},
Ringel's Conjecture holds for many classes of trees 
such as stars, paths, bistars, carterpillars, and lobsters (see~\cite{EdHo06,Golomb72}).
H\"aggkvist~\cite{Ha89} generalized Ringel's Conjecture for regular graphs
as follows.

\begin{conjecture}[Graham--H\"aggkvist, 1989]\label{conj:haggkvist}
	Let \(T\) be a tree with \(\ell\) edges.
	If \(G\) is a \(2\ell\)-regular graph, 
	then~\(G\) admits a \(T\)-decomposition
\end{conjecture}

H\"aggkvist~\cite{Ha89} also proved Conjecture~\ref{conj:haggkvist} 
when \(G\) has girth at least the diameter of \(T\).
For more results on decompositions of regular graphs into trees, 
see~\cite{Er15+,Fi90,JaTrTu91,JaKoWe13+}.
For the case where \(T=P_\ell\) is the path with \(\ell\) edges
(note that this notation is not standard), 
Kouider and Lonc~\cite{KoLo99} improved H\"aggkvist's result
proving that \textsl{if \(G\) is a \(2\ell\)-regular graph with girth \(g\geq (\ell+3)/2\),
then~\(G\) admits a \emph{balanced} \(P_\ell\)-decomposition~\(\D\)}, that
is a path decomposition \(\D\) where each vertex is the end-vertex
of exactly two paths of \(\D\). 
These authors also stated the following strengthening of Conjecture~\ref{conj:haggkvist}
for paths.
\begin{conjecture}[Kouider--Lonc, 1999]\label{conj:koulonc}
	Let \(\ell\) be a positive integer.
	If \(G\) is a \(2\ell\)-regular graph,
	then \(G\) admits a balanced \(P_\ell\)-decomposition.
\end{conjecture}

One of the authors~\cite{BoMoOsWa15+reg} proved the following weakening of Conjecture~\ref{conj:koulonc}:
for every positive integers~\(\ell\) and~\(g\) such that \(g\geq 3\),
there exists an integer \(m_0 = m_0(\ell,g)\) such that, 
if \(G\) is a \(2m\ell\)-regular graph with~\(m\geq m_0\),
then~\(G\) admits a \(P_\ell\)-decomposition \(\D\) such that every vertex of \(G\)
is the end-vertex of exactly~\(2m\) paths of~\(\D\).
In this paper we prove Conjecture~\ref{conj:koulonc} in the case \(\ell=4\).\\

\subsection{Notation}

A \emph{trail} \(T\) is a graph for which there is a sequence \(B = x_0\cdots x_\ell\)
of its vertices such that \(E(T) = \{x_ix_{i+1}\colon 0\leq i \leq \ell-1\}\)
and \(x_ix_{i+1} \neq x_jx_{j+1}\), for every \(i\neq j\).
Such a sequence~\(B\) of vertices is called a \emph{tracking} of \(T\)
and we say that \(T\) is the trail \emph{induced} by the tracking~\(B\).
We say that the vertices \(x_0\) and \(x_\ell\) are the \emph{final vertices} of \(B\).
Given a tracking \(B=x_0\cdots x_\ell\) we denote by \(B^-\) the tracking \(x_\ell\cdots x_0\).
By abuse of notation, 
we denote by~\(V(B)\) and~\(E(B)\) the sets \(\{x_0,\ldots,x_\ell\}\) of vertices,
and \(\{x_ix_{i+1}\colon 0\leq i\leq \ell-1\}\) of edges of~\(B\), respectively.
Moreover, we denote by \(\bar B\) the trail \(\big(V(B),E(B)\big)\),
and by \emph{length} of \(B\) we mean the length of \(\bar B\).
We also use \emph{\(\ell\)-tracking} to denote a tracking of length \(\ell\).
A set of edge-disjoint trackings \(\B\) of a graph \(G\) is a \emph{tracking decomposition} of \(G\)
if \(\cup_{B\in \B} E(B) = E(G)\).
If every tracking of \(\B\) has length \(\ell\), we say that \(\B\)
is an \emph{\(\ell\)-tracking decomposition},
and if every tracking of \(\B\) induces a path, 
we say that \(\B\) is a \emph{path tracking decomposition}.
For ease of notation, in this work we make no distinction between the trackings~\(B\) and~\(B^-\)
in the following sense.
Suppose \(B\in\B\) is a tracking of a trail~\(T\);
when we need to choose a tracking of~\(T\) we choose between~\(B\) and~\(B^-\) conveniently.

An \emph{orientation} \(O\) of a subset \(E'\) of edges of \(G\) 
is an attribution of a direction (from one vertex to the other)
to each edge of \(E'\).
If an edge \(xy\) is directed from \(x\) to \(y\) in~$O$,
we say that \(xy\) \emph{leaves} \(x\) and \emph{enters} \(y\).
Given a vertex \(v\) of~\(G\), 
we denote by \(d^+_O(v)\) \big(resp.~\(d^-_O(v)\)\big) 
the number of edges leaving (resp. entering) \(v\)
with respect to \(O\).
We say that \(O\) is \emph{Eulerian} if \(d^+_O(v) = d^-_O(v)\)
for every vertex \(v\) of \(G\).
We also denote by \(O^-\), called \emph{reverse orientation}, 
the orientation of \(E'\)
such that if \(xy\in E'\) is directed from \(x\) to \(y\) in \(O\),
then~\(xy\) is directed from \(y\) to \(x\) in \(O^-\).

Suppose that every tracking in \(\B\) has length at least~\(2\).
We consider an orientation \(O\) of a set of edges of \(G\) as follows.
For each tracking \(B = x_0\cdots x_\ell\) in \(\B\),
we orient \(x_0x_1\) from \(x_1\) to \(x_0\),
and \(x_{\ell-1}x_\ell\) from \(x_{\ell-1}\) to \(x_\ell\).
Given a vertex \(v\) of \(G\), we denote by \(\B(v)\) the number of edges 
of \(G\) directed towards \(v\) in \(O\) (i.e., \(\B(v) = d_O^-(v)\))
and by \(\prehang(v,\B)\) the number of edges leaving \(v\) in~\(O\) \big(i.e., \(\prehang(v,\B) = d_O^+(v)\)\big).
We say that an edge that leaves~\(v\) in~\(O\) 
is a \emph{hanging} edge at \(v\)
(this definition coincides with the definition of \emph{pre-hanging} edge in~\cite{BoMoOsWa15+thomassen}).
We say that a tracking decomposition \(\B\) of \(G\) is \emph{balanced}
if \(\B(u) = \B(v)\) for every \(u,v\in V(G)\).
It is clear that if \(\B\) is a balanced path tracking decomposition of \(G\), 
then \(\bar \B\) is a balanced path decomposition of \(G\).

We say that a subgraph \(F\) of a graph \(G\) is a \emph{factor} of~\(G\) if \(V(F) = V(G)\).
If a factor~\(F\) is \(r\)-regular, we say that \(F\) is an \emph{\(r\)-factor}.
Also, we say that a decomposition~\(\mathcal{F}\) of \(G\) is an \emph{\(r\)-factorization}
if every element of \(\mathcal{F}\) is an \(r\)-factor.

\subsection{Overview of the proof}

Let \(G\) be an \(8\)-regular graph.
In Section~\ref{section:extensions} we use Petersen's \(2\)-factorization theorem to obtain 
a \(4\)-factorization \(\{F_1,F_2\}\) of \(G\).
Then, we prove that \(F_1\) admits a balanced \(P_2\)-decomposition \(\D\).
Given an Eulerian orientation \(O\) to the edges of~\(F_2\),
we \emph{extend} each path \(P\) of \(\D\) 
to a trail of length \(4\) using one outgoing edge of~\(F_2\) at each end-vertex of \(P\)
(see Figure~\ref{extensions}),
thus obtaining a \(4\)-tracking decomposition~\(\B\) of~\(G\).
We also prove that these extensions can be chosen such that no element of \(\B\) is a cycle of length 4.
Lemma~\ref{lemma:good-orientation} shows that \(O\) can be chosen with some additional 
properties, which we call \emph{good orientation} (see Definition~\ref{def:good-orientation}),
and Lemma~\ref{lemma:path+exceptional} uses this special properties
to show that the elements of \(\B\) that do not induce paths
can be paired with paths of \(\B\)
to form a new special element, which we call \emph{exceptional extension}
(see Figure~\ref{fig:exceptionals}).
Thus, we can understand \(\B\) as a decomposition into paths and exceptional extensions.
In Section~\ref{section:complete}, we show how to switch edges between the elements
to obtain a decomposition into paths.

\section{Decompositions into extensions}\label{section:extensions}

In this section we use Petersen's Factorization Theorem~\cite{Pe1891} to obtain a well-structured 
tracking decomposition of \(8\)-regular graphs, called \emph{exceptional decomposition into extensions}.

\begin{theorem}[Petersen's $2$-Factorization Theorem]\label{thm:Pet}
	Every \(2k\)-regular graph admits a \(2\)-factorization.
\end{theorem}

Let \(G\) be an \(8\)-regular graph and let \(\mathcal{F}\) be a \(2\)-factorization of \(G\)
given by Theorem~\ref{thm:Pet}.
By combining the elements of \(\mathcal{F}\) we obtain a decomposition of \(G\)
into two \(4\)-factors, say \(F_1\) and \(F_2\).
From now on, we fix such two \(4\)-factors $F_1$ and $F_2$. 
In the figures throughout the paper, we color the edges of $F_1$ with red,
and the edges of $F_2$ with black.
We first prove the following straightforward lemma.

\begin{lemma}\label{lemma:4-reg->2-path}
	If \(G\) is a \(4\)-regular graph, 
	then \(G\) admits a balanced \(P_2\)-decomposition.
\end{lemma}

\begin{proof}
Let \(G\) be a \(4\)-regular graph and 
fix an Eulerian orientation \(O\) of its edges.
For each vertex \(v\) of \(G\),
let \(P_v\) be the path consisting of the two edges of \(G\) that leave~\(v\) in \(O\).
The set \(\big\{P_v\colon v\in V(G)\big\}\) is a balanced \(P_2\)-decomposition of~\(G\).
\end{proof}

Now, let \(\D_1\) be a balanced \(P_2\)-decomposition of \(F_1\),
\(O\) be an orientation of the edges of \(F_2\),
and \(B = x_0x_1x_2x_3x_4\) be a \(4\)-tracking in \(G\).
We say that \(B\) is a \emph{\((\D_1,O)\)-extension} if \(x_1x_2x_3 \in \D_1\),
\(x_1x_0\) is directed from~\(x_1\) to~\(x_0\),
and \(x_3x_4\) is directed from~\(x_3\) to~\(x_4\).
We note that if \(T\) is a \((\D_1,O)\)-extension,
then exactly one of the following holds: 
(a) \(T\) is a path of length \(4\);
(b) \(T\) contains a triangle;
(c) \(T\) is a cycle of length \(4\)
(see Figure~\ref{extensions}).
We say that a tracking decomposition \(\B\) of \(G\)
is a \emph{decomposition into \((\D_1,O)\)-extensions} 
if every element of \(\B\) is a \((\D_1,O)\)-extension.
We omit \(\D_1\) and \(O\) when it is clear from the context. 
The next result shows that every \(8\)-regular graph 
admits a decomposition into extensions with no cycles.
We denote by \(\tau(\B)\) the number of elements of \(\B\)
that are cycles of length~\(4\).

\begin{figure}[h]\centering
	\begin{subfigure}[b]{.33\linewidth}
	\centering\scalebox{1}{\input{extension1}}
	\caption{}\label{fig:extension-a}
	\end{subfigure}%
	\begin{subfigure}[b]{.33\linewidth}
	\centering\scalebox{1}{\input{extension2}}
	\caption{}\label{fig:extension-b}
	\end{subfigure}%
	\begin{subfigure}[b]{.33\linewidth}
	\centering\scalebox{1}{\input{extension3}}
	\caption{}\label{fig:extension-c}
	\end{subfigure}%
	\caption{Extensions}
\label{extensions}
\end{figure}
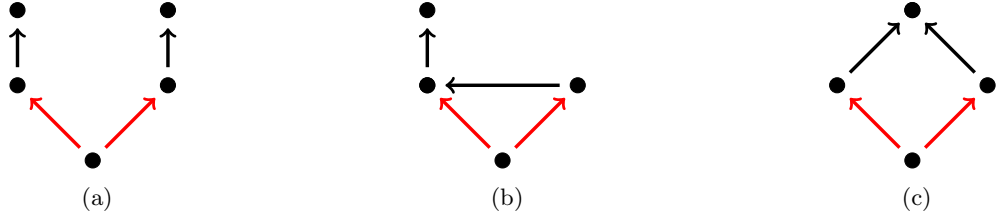

\begin{lemma}\label{lemma:8-reg->extensions}
	Let \(G\) be an \(8\)-regular graph,
	\(F\) be a \(4\)-factor of \(G\),
	\(\D\) be a balanced \(P_2\)-decomposition of~\(F\),
	and \(O\) be an Eulerian orientation of the edges of \hbox{\(G-E(F)\).}
	Then~\(G\) admits a decomposition into \((\D,O)\)-extensions
	with no cycles.
\end{lemma}

\begin{proof}
	Let \(G\), \(\D\), and \(O\) be as in the statement,
	and put \(H = G-E(F)\).
	First, we prove that \(G\) admits a decomposition
	into \((\D,O)\)-extensions.
	Indeed, since \(\D\) is balanced, \(\D(v) = 2 = d^+_O(v)\) 
	for every vertex \(v\) of \(G\).
	Thus, we can extend every path \(P=x_1x_2x_3\) in \(\D\) 
	to a \((\D,O)\)-extension \(Q_P = x_0x_1x_2x_3x_4\) such that
	\(x_0x_1\) and \(x_3x_4\) are edges leaving \(x_1\) and \(x_3\), respectively,
	and such that every edge of \(H\) is used exactly once.
	Therefore, \(\{Q_P\colon P\in\D\}\) is a decomposition into \((\D,O)\)-extensions.
	
	Now, let \(\B\) be a decomposition of \(G\) into \((\D,O)\)-extensions 
	that minimizes \(\tau(\B)\).
	Suppose, for contradiction, that \(\tau(\B) > 0\).
	Let \(T = x_0x_1x_2x_3x_4\) be 
	a cycle of length \(4\) in \(\B\),
	where \(x_1x_2x_3 \in \D\) and \(x_0 = x_4\).
	Let \(B = y_1y_2y_3\) be an element of \(\D\) 
	such that \(B\neq T\) and \(y_1 = x_1\).
	Let \(Q = y_0y_1y_2y_3y_4\) be the element of~\(\B\) 
	that contains \(B\),
	and put \(T' = y_0x_1x_2x_3x_4\) and \(Q' = x_0y_1y_2y_3y_4\).
	Clearly,~\(T'\) and \(Q'\) are \((\D,O)\)-extensions,
	and \(T'\) is not a cycle.
	Moreover, if \(Q'\) is a cycle, 
	then the edges \(x_0x_1\), \(x_3x_4\), and \(y_3y_4\) are 
	directed towards \(x_0\), which implies \(d_O^-(x_0) \geq 3\),
	hence \(O\) is not an Eulerian orientation, a contradiction.
	Therefore, \(\B' = \B - T + T' - Q + Q'\)
	is a decomposition into \((\D,O)\)-extensions
	such that \(\tau(\B') \leq \tau(\B) -1\),
	a contradiction to the minimality of~\(\tau(\B)\).
\end{proof}

The following fact about decompositions into extensions are used
in Section~\ref{section:complete}.

\begin{fact}\label{fact:balanced+prehang}
	Let \(G\) be an \(8\)-regular graph,
	\(F\) be a \(4\)-factor of \(G\),
	\(\D\) be a balanced \(P_2\)-decomposition of \(F\),
	\(O\) be an Eulerian orientation of the edges of \(G-E(F)\),
	and \(\B\) be a decomposition of \(G\)
	into \((\D,O)\)-extensions.
	Then \(\B\) is balanced and \(\prehang(v,\B) =2\)
	for every vertex \(v\) of \(G\).
\end{fact}

\begin{proof}
	Let \(G\), \(\D\), \(O\), and \(\B\) be as in the statement,
	and put \(H = G-E(F)\).
	Since \(O\) is an Eulerian orientation of \(H\),
	\(d^+_O(v) = d^-_O(v) = 2\) for every vertex
	\(v\) of \(G\).
	By the definition of \(\B(v)\),
	\(\B(v) = d^-_O(v) = 2\) for every vertex \(v\) of \(G\).
	Therefore, \(\B\) is balanced.
	By the definition of \(\prehang(v,\B)\),
	\(\prehang(v,\B) = d^+_O(v) = 2\)
	for every vertex \(v\) of \(G\).
\end{proof}

\subsection{Trapped subgraphs and good orientation}

In this subsection we define two special concepts, namely, \emph{trapped subgraphs} and \emph{good orientations}, that are used throughout this section.

We say that 
an edge $uv \in F_2$ is \emph{trapped} by $\D_1$ if there exists 
a path $P \in \D_1$ whose end-vertices are precisely~$u$ and~$v$.
Alternatively, we say that \(P\) \emph{traps} the edge \(uv\).
Moreover, we say that
	
			an induced path \(uvw\) in \(G[F_2]\) is a \emph{$\D_1$-trapped \(P_2\)} 
			if the edges $uv$ and $vw$ are trapped by $\D_1$ and there exists 
			a path in $\D_1$ whose end-vertices are precisely $u$ and $w$
			(see Figure~\ref{fig:trapped-a});
			a triangle \(uvwu\) in $G[F_2]$ is a \emph{$\D_1$-trapped triangle} 
			(resp. \emph{$\D_1$-quasi-trapped triangle}) if all its edges 
			(resp. two of its edges) are trapped by~$\D_1$
			(see Figure~\ref{fig:trapped-b} and~\ref{fig:trapped-c}); 
			and a copy~\(H\) of~\(K_4\) in~\(G[F_2]\) is a \emph{$\D_1$-trapped \(K_4\)} 
			if four of its edges are trapped by \(\D_1\)
			(see Figure~\ref{fig:trapped-d}).
			We omit the decomposition \(\D_1\) when it is clear from the context.
			By a \emph{trapped subgraph} of \(G\) 
			we mean a subgraph of \(G[F_2]\) that is a trapped \(P_2\), 
			trapped triangle, or trapped \(K_4\).
			If a trapped edge \(e\) is not contained in any trapped subgraph or quasi-trapped triangle,
			then we say that \(e\) is a \emph{free} trapped edge.

\begin{figure}[h]\centering
	\begin{subfigure}[b]{.25\linewidth}
	\centering\scalebox{.8}{\input{trapped-P2}}
	\caption{}\label{fig:trapped-a}
	\end{subfigure}%
	\begin{subfigure}[b]{.25\linewidth}
	\centering\scalebox{.8}{\input{trapped-triangle}}
	\caption{}\label{fig:trapped-b}
	\end{subfigure}%
	\begin{subfigure}[b]{.25\linewidth}
	\centering\scalebox{.8}{\input{quasi-trapped-triangle-first}}
	\caption{}\label{fig:trapped-c}
	\end{subfigure}%
	\begin{subfigure}[b]{.25\linewidth}
	\centering\scalebox{.8}{\input{trapped-K4}}
	\caption{}\label{fig:trapped-d}
	\end{subfigure}%
	\caption{Trapped subgraphs of $F_2$}\label{fig:1}
\label{trapped-ex}
\end{figure}
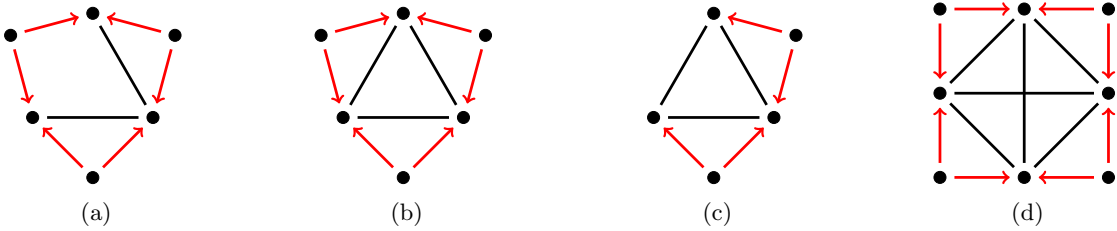

Let \(T\) be a trapped \(P_2\) or quasi-trapped triangle of \(G\),
where \(uv\) and \(vw\) are the trapped edges of \(T\).
We say that an orientation \(O\) of the edges of \(T\)
is \emph{consistent} if \(d^+_O(v) = d^-_O(v)\),
otherwise, we say that \(O\) is \emph{centered}.
Now, we are able to define our special Eulerian orientation.

\begin{definition}\label{def:good-orientation}
	Let \(G\) be an \(8\)-regular graph,
	\(F\) be a \(4\)-factor of \(G\),
	\(\D\) be a balanced \(P_2\)-decomposition of \(F\).
We say that an Eulerian orientation \(O\) of the edges of \(G-E(F)\)
is \emph{good} if the following hold.
\begin{itemize}
\item[(i)] 	If \(T\) is a trapped \(P_2\) of \(G\),
		then \(O\) induces a consistent orientation 
		of the edges of~\(T\);
\item[(ii)]	if \(T\) is a trapped triangle of \(G\),
		then \(O\) induces an Eulerian orientation
		of the edges of~\(T\); and

\item[(iii)]	
		if \(T\) is a quasi-trapped triangle of \(G\), 
		then \(O\) induces an Eulerian orientation
		or a centered orientation
		of the edges of~\(T\)
		(see Figure~\ref{fig:quasi-good}).
\end{itemize}
\end{definition}

\begin{figure}[h]\centering
	\begin{subfigure}[b]{.33\linewidth}
	\centering\scalebox{1}{\input{quasi-eulerian}}
	\end{subfigure}%
	\begin{subfigure}[b]{.33\linewidth}
	\centering\scalebox{1}{\input{quasi-centered1}}
	\end{subfigure}%
	\begin{subfigure}[b]{.33\linewidth}
	\centering\scalebox{1}{\input{quasi-centered2}}
	\end{subfigure}%
	\caption{Eulerian and centered orientations of quasi-trapped triangles.}\label{fig:quasi-good}\vspace{-.5cm}
\end{figure}
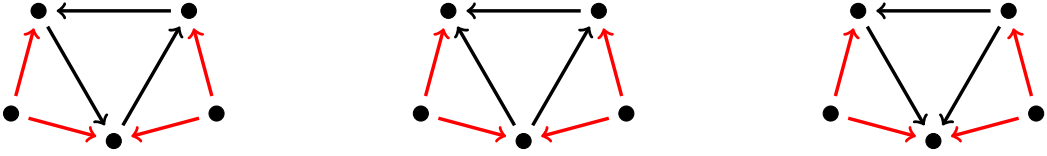

Note that, since \(\D_1\) is balanced, \(\D_1(v) = 2\)
for every \(v\in V(G)\).
Since each path \(P\) in \(\D_1\) traps at most one edge of \(F_2\)
and \(\D_1(v) = 2\) for every vertex \(v\) of \(G\),
each vertex \(v\) of \(G\) is incident to at most two trapped edges.
Therefore, the subgraph \(F_2^t\) of \(F_2\) induced by the \(\D_1\)-trapped edges of \(F_2\)
is composed of vertex-disjoint paths and cycles.
This implies the following fact.
Two quasi-trapped triangles can have one edge 
in common (see Figure~\ref{chain-quasi-trapped-ex}),
but each edge of \(F_2\) is contained 
in at most one trapped subgraph of \(F_2\).
Indeed, let \(T\) be a trapped subgraph of \(G\),
and \(T^t\) be the subgraph of \(F_2\)
induced by the trapped edges of~\(T\).
Clearly, \(T^t\) is a subgraph of \(F_2^t\).
If \(T\) is a trapped triangle (resp. trapped \(K_4\)), 
then \(T^t\) is a triangle (resp. a cycle of length \(4\)), 
hence \(T^t\) is a component of~\(F_2^t\).
If \(T\) is a trapped~\(P_2\), say \(T = uvw\), 
then \(T^t\) is the path \(uvw\),
and \(d_{F_2^t}(u) = d_{F_2^t}(w) = 1\).
Therefore,~\(T^t\) is a component of~\(F_2^t\).
Since two components of \(F_2^t\) do not intercept, 
each edge of \(F_2\) is contained in at most one trapped subgraph of \(G\).
Moreover, if \(T\) is a quasi-trapped triangle and intercepts
a trapped subgraph~\(T'\),
then \(T'\) must be a trapped~\(K_4\).

In what follows we study sequences of quasi-trapped triangles. 
Note that two distinct quasi-trapped triangles have at most one trapped edge in common.
We say that a sequence \(S=T_1\cdots T_k\) of quasi-trapped triangles
is a \emph{chain} of quasi-trapped triangles
if \(T_i\) and \(T_{i+1}\) have a trapped edge in common for \(i=1,\ldots,k-1\).
If \(T_k\) and \(T_1\) also have a trapped edge in common,
then we say that \(S\) is a \emph{closed} chain of quasi-trapped triangles,
otherwise we say that \(S\) is \emph{open}.
The following fact about chains of quasi-trapped triangles are used in the proof of 
Lemma~\ref{lemma:good-orientation}.

\begin{fact}\label{fact:quasi-trapped-chain}
Let \(S=T_1\cdots T_k\) be a chain of quasi-trapped triangles,
and put \(G_S = \cup_{i=1}^k T_i\).
If~\(S\) is open, then the trapped edges of \(G_S\) induce 
a Hamiltonian path \(P = a_0\cdots a_{k+1}\) in~\(G_S\),
where \(a_1\) and \(a_k\) are precisely the two vertices of odd degree in \(G_S\).
If \(S\) is closed then \(k>3\) and the following hold.
\begin{itemize}
	\item[(i)]	If \(k=4\), then \(G_S\) is a trapped \(K_4\); and
	\item[(ii)]	If \(k>4\), then \(G_S\) is a \(4\)-regular graph
			and the trapped edges of \(G_S\) 
			induce a Hamiltonian cycle in \(G_S\)
\end{itemize}
\end{fact}

We conclude that the subgraph \(F_2^t\) of \(F_2\) induced by the trapped edges of \(F_2\) can be decomposed into free trapped edges, trapped \(P_2\)'s, trapped triangles, trapped~\(K_4\)'s,
and maximal sequences of quasi-trapped triangles (that are not contained in trapped \(K_4\)'s).

\begin{lemma}\label{lemma:good-orientation}
Let \(G\) be an \(8\)-regular graph, 
\(F\) be a \(4\)-factor of \(G\), 
\(\D\) be a balanced \(P_2\)-decomposition of~\(F\),
and put \(H = G -E(F)\).
Then, there is a good Eulerian orientation of the edges of \(H\).
\end{lemma}

\begin{proof}

Let \(G\), \(\D\), \(F\), and \(H\) be as in the statement.
In what follows, we construct a new Eulerian graph \(H^*\) from \(H\),
then we use an Eulerian orientation of \(H^*\) to obtain a good Eulerian orientation of \(H\).

First, we deal with trapped subgraphs, and then with (chains of) quasi-trapped triangles.
For every trapped~\(P_2\), say \(T = uvw\), where \(uv\) and \(vw\) are trapped edges,
we split edges in the following way.
We add a new vertex \(z_T\), delete the edges \(uv\) and \(vw\),
and add the edges \(uz_T\) and \(z_Tw\). 
For every trapped triangle or trapped \(K_4\), say \(T\),
delete all the trapped edges of \(T\).
It is clear that the graph \(H'\) obtained after these operations is Eulerian.

Now, let \(S= T_1\cdots T_k\) be a maximal chain of quasi-trapped triangles in \(H'\),
and put \(G_S=\cup_{i=1}^k T_i\).
If \(S\) is closed then \(k > 4\), because \(H'\) has no trapped \(K_4\).
By Fact~\ref{fact:quasi-trapped-chain}, 
\(G_S\) is a \(4\)-regular subgraph of \(H'\)
and we delete the edges of \(G_S\).
Now, suppose that \(S\) is open.
If~\(k=1\), then delete the edges of \(T_1\).
If~\(k>1\), then by Fact~\ref{fact:quasi-trapped-chain},
\(G_S\) contains a Hamiltonian path induced 
by the trapped edges in \(G_S\), say \(P = a_0\cdots a_{k+1}\),
where~\(a_1\) and~\(a_k\) have odd degree in \(G_S\).
In this case, we delete the edges of~\(G_S\), and add the edge~\(a_1a_k\).

It is clear that the graph \(H^*\) obtained after these operations is again Eulerian.
Therefore, let \(O^*\) be an Eulerian orientation of the edges of \(H^*\).
In what follows, we ``undo'' the operations above and obtain a good orientation \(O\) of the edges of \(H\).

We must show how to orient each edge of \(H\).
If \(e\) is an edge in \(E(H)\cap E(H^*)\) that is not contained in any trapped \(K_4\) of \(G\),
then we direct \(e\) in \(O\) with the same direction \(e\) has in \(O^*\).
Let \(S= T_1\cdots T_k\) be a maximal chain of quasi-trapped triangles in \(H\), and put \(G_S=\cup_{i=1}^k T_i\).
If \(k=1\), 
then \(G_S\) is a triangle, and by the definition of \(H^*\), no edge of \(G_S\) is in \(H^*\),
and we give an Eulerian orientation to the edges of \(G_S\).
If \(k>4\) and \(S\) is a closed chain,
then \(S\) is not a trapped \(K_4\).
By the definition of \(H^*\), no edge of \(G_S\) is in \(H^*\).
By Fact~\ref{fact:quasi-trapped-chain}, \(G_S\) is \(4\)-regular and the trapped edges
of \(G_S\) induce a Hamiltonian cycle \(C_S = a_0\cdots a_{k-1}a_0\) in \(G_S\).
Note that the edges in \(G_S - E(C_S)\) are precisely \(a_ia_{i+2}\) for \(i=0,\ldots,k-1\),
where \(a_k = a_0\) and \(a_{k+1} = a_1\).
Thus, (in \(O\)) orient \(a_ia_{i+1}\) from \(a_i\) to \(a_{i+1}\),
and \(a_ia_{i+2}\) from \(a_{i+2}\) to \(a_i\), for \(i=0,\ldots,k-1\).
Note that the triangle \(a_ia_{i+1}a_{i+2}\) has an Eulerian orientation, for \(i=0,\ldots,k-1\).
Now, suppose that \(k\geq 2\) and \(S\) is an open chain.
By Fact~\ref{fact:quasi-trapped-chain},
\(G_S\) contains a Hamiltonian path, say \(P = a_0\cdots a_{k+1}\),
where \(a_1\) and \(a_k\) have odd degree in \(G_S\).
From the construction of \(H^*\), we have \(a_1a_k\in E(H^*)\).
Let \(O_S\) be the orientation of \(G_S\) where the edge \(a_ia_{i+1}\) 
is oriented from \(a_{i}\) to \(a_{i+1}\), for \(i = 0,\ldots,k\)
and \(a_ia_{i+2}\) is oriented from \(a_{i+2}\) to \(a_i\),
for \(i = 0,\ldots,k-1\) (see Figure~\ref{chain-quasi-trapped-ex}).
Note that the triangle \(a_ia_{i+1}a_{i+2}\) has an Eulerian orientation in \(O_S\).
If \(a_1a_k\) is oriented from \(a_k\) to \(a_1\) in \(O^*\),
then orient in \(O\) the edges of \(G_S\) according to \(O_S\).
Otherwise \(a_1a_k\) is oriented from \(a_1\) to \(a_k\) in \(O^*\),
and we orient in \(O\) the edges of \(G_S\) according to \(O_S^-\).
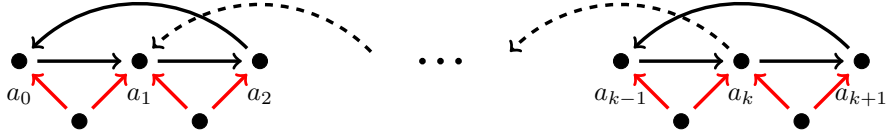
\begin{figure}[h]\centering
\scalebox{1}{\input{quasi-trapped-triangle-new2.tex}}
\caption{Eulerian orientation of quasi-trapped triangles.}
\label{chain-quasi-trapped-ex}
\end{figure}
We have just chosen a direction to every edge in \(E(H')\setminus E(H^*)\), 
except the (not trapped) edges in trapped~\(K_4\)'s.
Thus, if \(T\) is a quasi-trapped triangle in \(H\) not contained
in a trapped \(K_4\), then \(O\) induces an Eulerian orientation of 
the edges of \(T\).

Now, let \(K\) be a trapped \(K_4\),
and let \(x_ix_{i+1}\) be the trapped edges of \(K\),
for \(i=0,1,2,3\), where \(x_4 = x_0\).
By construction, \(H^*\) contains the edges \(x_0x_2\) and \(x_1x_3\).
Suppose, without loss of generality, that in \(O^*\) 
the edge \(x_0x_2\) is directed from \(x_0\) to \(x_2\),
and the edge \(x_1x_3\) is directed from \(x_1\) to \(x_3\).
We orient the edges of \(K\) in \(O\) in the following way.
The edge \(x_0x_2\) is directed from \(x_0\) to \(x_2\),
and the edge \(x_1x_3\) is directed from \(x_3\) to \(x_1\)
(i.e., with the direction opposite to the direction of \(x_1x_3\) in \(O^*\)).
Moreover, orient the trapped edges of \(K\) such that
\(x_1x_2x_3\) and \(x_1x_0x_3\) are two directed paths from \(x_1\) to~\(x_3\),
i.e.,
\(x_ix_{i+1}\) is directed from \(x_i\) to \(x_{i+1}\) for \(i=1,2\),
and directed from \(x_{i+1}\) to \(x_i\), for \(i=0,3\)
(see Figure~\ref{fig:K4-orientation}).
Note that the orientations induced by \(O\) 
of the triangles \(x_0x_1x_3\) and \(x_1x_2x_3\) are Eulerian
and that the orientations induced by \(O\) 
of the triangles \(x_0x_2x_3\) and \(x_0x_1x_2\) are centered
(see Item (iii) of Definition~\ref{def:good-orientation}).
We have chosen a direction to every edge in \(E(H)\cap E(H')\).
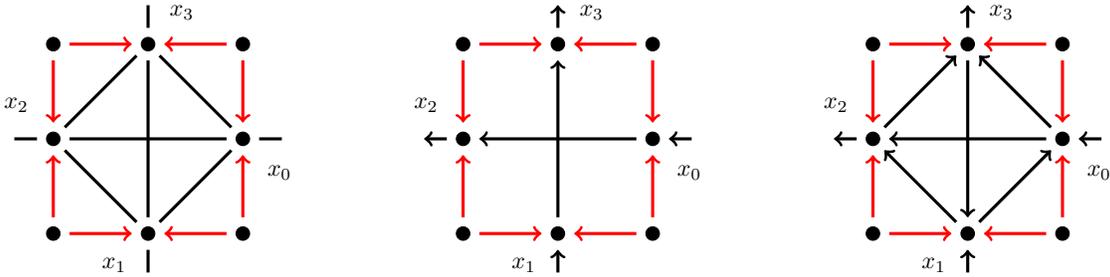
\begin{figure}[h]\centering
	\begin{subfigure}[b]{.33\linewidth}
	\centering\scalebox{.9}{\input{trapped-K4-orientation1}}
	\end{subfigure}%
	\begin{subfigure}[b]{.33\linewidth}
	\centering\scalebox{.9}{\input{trapped-K4-orientation3}}
	\end{subfigure}%
	\begin{subfigure}[b]{.33\linewidth}
	\centering\scalebox{.9}{\input{trapped-K4-orientation4}}
	\end{subfigure}%
	\caption{Good orientation of a trapped \(K_4\).}\vspace{-.2cm}
\label{fig:K4-orientation}
\end{figure}

If \(T\) is a trapped triangle in \(H\),
then in \(O\) we orient the edges of \(T\) with any Eulerian orientation
(see Item (ii) of Definition~\ref{def:good-orientation}).
Finally, let \(T=uvw\) be a trapped~\(P_2\) in~\(H\).
There exists a vertex \(z_T\) in \(H'\) incident exactly
to the edges \(uz_T\) and \(z_Tw\).
Thus \hbox{\(d_{O^*}^+(z_T) = d_{O^*}^-(z_T)=1\).} 
If~\(uz_T\) is directed from \(u\) to \(z_T\),
we orient \(uv\) from~\(u\) to~\(v\),
and \(vw\) from \(v\) to \(w\) in \(O\);
otherwise, we orient~\(uv\) from~\(v\) to~\(u\),
and \(vw\) from \(w\) to \(v\) in \(O\).
This gives a consistent orientation to every trapped \(P_2\) in \(H\)
(see Item (i) of Definition~\ref{def:good-orientation}).
We conclude that \(O\) is a good Eulerian orientation of \(H\).
\end{proof}

\subsection{Double-trapped edges and exceptional extensions}

We say that an edge \(uv\in F_2\) is \emph{double-trapped} by~\(\D_1\)
if there exist two distinct paths in \(\D_1\)
whose end-vertices are precisely \(u\) and \(v\).
Let \(e\in F_2\) be double-trapped by~\(\D_1\).
If~\(P_1\) and~\(P_2\) are the paths of~\(\D_1\) that trap~\(e\), 
then \(P_1+e\) and \(P_2+e\) are triangles of~\(G\).
Thus, if an edge \(e\in F_2\) is double-trapped then it is the case
for any orientation of \(F_2\).
Therefore, if \(\B\) is a decomposition of \(G\) into \((\D_1,O)\)-extensions
for some Eulerian orientation~\(O\) of~\(F_2\), 
and~\(T\) is the element of \(\B\) that contains \(e\),
then \(T\) contains a triangle.

Our next goal is to show that every \(8\)-regular graph \(G\)
admits a decomposition into paths of length \(4\)
and a special object which we call \emph{exceptional extension}.

Let \(e\in F_2\) be double-trapped by the paths~\(P_1\) and~\(P_2\)
of \(\D_1\), \(O\) be an Eulerian orientation of \(F_2\),
and~\(\B\) be a decomposition into \((\D_1,O)\)-extensions of \(G\) with no cycles.
Let \(T_i\) be the element of~\(\B\) that contains~\(P_i\),
for~\(i=1,2\).
The \emph{exceptional extension} that contains~\(e\) is the pair
\(X = \{T_1,T_2\}\) (see Figure~\ref{fig:exceptionals}).
It is clear that an exceptional extension contains a path of length \(4\) and a trail of length \(4\) that contains a triangle.
We say that a decomposition into extensions \(\B\) with no cycles
is \emph{exceptional}
if every element \(T\in\B\) that contains a triangle
is contained in exactly one exceptional extension.

\begin{figure}[h]\centering
	\begin{subfigure}[b]{.5\linewidth}
	\centering\scalebox{1}{\input{exceptional}}
	\end{subfigure}%
	\begin{subfigure}[b]{.5\linewidth}
	\centering\scalebox{1}{\input{exceptional2}}
	\end{subfigure}%
\caption{Exceptional extensions}\label{fig:exceptionals}
\end{figure}
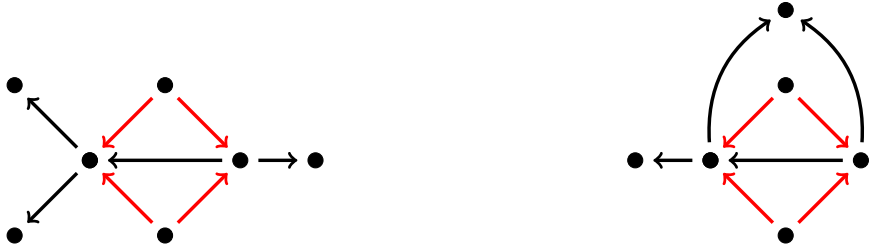

Now we are able to prove the main result of this section.
Given a \(4\)-tracking decomposition \(\B\),
we denote by \(\tau'(\B)\) the number of elements of \(\B\)
that are paths.

\begin{lemma}\label{lemma:path+exceptional}
	Let \(G\) be an \(8\)-regular graph, 
	\(F\) be a \(4\)-factor of \(G\), 
	\(\D\) be a balanced \(P_2\)-decomposition of~\(F\),
	and \(O\) be a good Eulerian orientation of the edges of \(G-E(F)\).
	Then~\(G\) admits an exceptional decomposition into \((\D,O)\)-extensions.
\end{lemma}

\begin{proof}
	Let \(G\), \(F\), \(\D\), and \(P\) be as in the statement,
	By Lemma~\ref{lemma:8-reg->extensions},
	there is a decomposition~\(\B\) of~\(G\) 
	into \((\D,O)\)-extensions with no cycles.
	Let~\(\B\) be a decomposition of \(G\) 
	into \((\D,O)\)-extensions with no cycles
	that maximizes~\(\tau'(\B)\).
	We claim that if there is an element \(T\) in \(\B\)
	that contains a triangle, 
	then \(T\) contains a double-trapped edge.
	
	Suppose, for a contradiction, that there is an element of \(\B\), 
	that contains a triangle.
	Let \(T = x_0x_1x_2x_3x_4\in\B\cup\B^-\),
	where \(x_1x_2x_3 \in \D\), \(x_4 = x_1\),
	and \(x_3x_4\) is not double-trapped.
	Let \(Q = y_0y_1y_2y_3y_4\) be an element of \(\B\cup\B^-\)
	with \(Q\neq T\), and
	such that \(y_1y_2y_3\in\D\) and \(y_3 = x_3\),
	and put \(T' = x_0x_1x_2x_3y_4\), \(Q' = y_0y_1y_2y_3x_4\).
	Clearly, \(T'\) is a path or a cycle.
	Moreover, 
	since \(x_3x_4\) is not double-trapped,
	\(y_1y_2y_3x_4\) is not a triangle,
	hence \(Q'\) contains a triangle 
	only if \(Q\) contains the triangle \(y_0y_1y_2y_3\).
	If \(T'\) and \(Q'\) are not cycles, 
	then \(\B' = \B-T+T'-Q+Q'\) is a 
	decomposition into \((\D,O)\)-extensions with no cycles, 
	and \(\tau'(\B') = \tau'(\B)+1\),
	a contradiction to the maximality of \(\tau'(\B)\).
	In what follows we divide the proof on whether~\(T'\) and \(Q'\) are cycles.
	
	{\smallskip\noindent\bf Case 1: \(T'\) is a cycle and \(Q'\) is not a cycle.}
	In this case, 
	we have \(x_0 = y_4\) and \(x_0x_1x_3y_4\)
	is a triangle in \(H\).
	Thus, the edge \(x_0x_1\) is not trapped, 
	otherwise \(x_0x_1x_3y_4\) is either a quasi-trapped triangle
	without Eulerian or centered orientation,
	or a trapped triangle without Eulerian orientation.
	Let \(R = z_0z_1z_2z_3z_4\) an element in \(\B\cup\B^-\)
	different from \(T\),
	where \(z_1z_2z_3 \in \D\) and \(z_1 = x_1\).
	Since \(x_0x_1\) is not trapped, we have \(x_0 \neq z_3\).
	Moreover, we have \(z_4 \neq x_0\), 
	otherwise \(d^-_O(x_0) \geq 3\).
	Therefore, \(R' = x_0z_1z_2z_3z_4\) is a path.
	Also, since \(G\) has no parallel edges,
	we have \(z_0\neq x_3\) and \(z_0 \neq y_4\).
	Thus, \(T'' = z_0x_1x_2x_3y_4\) is a path.
	Therefore \(\B' = \B-T+T''-Q+Q'-R+R'\) is a 
	decomposition into \((\D,O)\)-extensions with no cycles, 
	and \(\tau'(\B') = \tau'(\B)+1\),
	a contradiction to the maximality of \(\tau'(\B)\).
	
	{\smallskip\noindent\bf Case 2: \(Q'\) is a cycle and \(T'\) is not a cycle.}
	Let \(U = w_0w_1w_2w_3w_4\) be an element in \(\B\cup\B^-\)
	different from~\(Q\),
	where \(w_1w_2w_3\in\D\) and \(w_1 = y_1\).
	If \(w_0 = y_3\), then \(x_4x_3y_1y_0\) is either a quasi-trapped triangle
	without Eulerian or centered orientation,
	or a trapped triangle without Eulerian orientation.
	Moreover, \(w_0\neq x_4\) because \(G\) has no parallel edges.
	Therefore, \(Q'' = w_0y_1y_2y_3x_4\) is a path.
	Now, let \(U' = y_0w_1w_2w_3w_4\).
	If \(U'\) contains the triangle \(y_0w_1w_2w_3\), 
	then \(w_1w_2w_3\) traps \(y_1y_0\) 
	and \(y_1y_0x_3\) is either a trapped \(P_2\),
	without consistent orientation,
	or a trapped triangle without Eulerian orientation.
	Therefore, \(U'\) contains a triangle 
	only if \(U\) contains the triangle \(w_1w_2w_3w_4\).	
	If \(U'\) is a cycle, then we have \(w_4 = y_0\) and 
	the edges \(x_3x_4\), \(y_1y_0\), and \(w_3w_4\)
	are directed toward~\(x_1\), hence \(d^-_O(x_1) \geq 3\).
	Thus \(U'\) is not a cycle.
	Therefore \(\B' = \B-T+T'-Q+Q''-U+U'\) is a 
	decomposition into \((\D,O)\)-extensions with no cycles, 
	and \(\tau'(\B') = \tau'(\B)+1\),
	a contradiction to the maximality of~\(\tau'(\B)\).

	{\smallskip\noindent\bf Case 3: \(T'\) and \(Q'\) are cycles.}	
	In this case we have \(x_4\neq y_1\), 
	otherwise \(y_0y_1\) and \(x_0x_1\) are parallel edges. 
	Let \(R = z_0z_1z_2z_3z_4\) and \(U = w_0w_1w_2w_3w_4\) be elements in \(\B\cup\B^-\)
	different from \(T\) and \(Q\),
	where \(z_1z_2z_3,w_1w_2w_3 \in \D\), and \(z_1 = x_1\) and \(w_1 = y_1\).
	We claim that \(R\neq U\).
	Indeed, if \(R = U\), then \(z_1z_2z_3 = w_1w_2w_3\)
	and \(y_1y_0\) is a trapped edge.
	Thus, \(y_1y_0x_3\) is a trapped \(P_2\),
	without consistent orientation.
	Thus, analogously to cases~\(1\) and~\(2\),
	\(R' = x_0z_1z_2z_3z_4\),
	\(T'' = z_0x_1x_2x_3y_4\),
	\(Q'' = w_0y_1y_2y_3x_4\) are paths,
	and \(U' = y_0w_1w_2w_3w_4\) is not a cycle.
	Put \(\B' = \B - T + T'' - Q + Q'' - R + R' - U + U'\).
	Again, \(\B\) is a decomposition into \((\D,O)\)-extensions 
	with no cycles
	such that \(\tau'(\B') \geq \tau'(\B) + 1\),
	a contradiction to the maximality of \(\tau'(\B)\).
	
	We conclude that every element \(T\) of \(\B\) 
	that contains a triangle 
	contains a double-trapped edge, say \(e_T\).
	Suppose that \(P_1\) and \(P_2\) are the elements of \(\D\)
	that trap \(e_T\), where \(P_1\subset T\),
	and let \(T'\) be the element of \(\B\) that contains \(P_2\).
	The pair \(X_T = \{T,T'\}\) is the exceptional extension 
	that contains \(e_T\).
	Thus, for every element~\(T\) of~\(\B\) 
	that contains a triangle
	we obtain an exceptional extension \(X_T\).
	It is clear that this exceptional extension is unique.
	Therefore, \(\B\) is exceptional.	
\end{proof}

\section{Complete decompositions}\label{section:complete}

In this section we relax the properties of the decomposition given
by Lemma~\ref{lemma:path+exceptional},
and prove that every \(8\)-regular graph 
admits a \(P_4\)-decomposition.
We start with the decomposition given 
by Lemma~\ref{lemma:path+exceptional}, 
and switch edges between the elements of this decomposition to obtain a decomposition
containing only paths.
Thus, the elements of the decompositions we consider here
do not depend on \(\D_1\) and \(O\).

First we give some definitions.
Let \(G\) be an 8-regular graph, 
and let \(\B\) be a \(4\)-tracking decomposition of \(G\) with no cycles.
An \emph{exceptional pair} of \(\B\) is 
a pair of elements \(X = \{T_1,T_2\}\) of \(\B\cup\B^-\) 
such that \(\bar T_1 \neq \bar T_2\), \(T_1 = a_1b_1c_1d_1e_1\), \(T_2 = a_2b_2c_2d_2e_2\),
and \(e_2 = b_2 = b_1\) and \(d_1=d_2\).
We say that the vertices \(c_1\) and \(c_2\) are the 
\emph{connection vertices} of \(X\), 
and that \(d_2\) is the \emph{center} of~\(X\).
By the definition of hanging edge, 
the edges \(a_1b_1\) and \(a_2b_2\) are hanging edges at~\(b_1\),
and \(d_1e_1\) and \(d_2e_2\) are hanging edges at \(d_1\).
Note that an exceptional pair is exactly a pair of the underlying graphs
of an exceptional extension.

The following definition presents the properties of the decompositions
given by Lemma~\ref{lemma:path+exceptional} that are used in 
the proof of our main result.

\begin{definition}\label{def:complete}
Let \(G\) be an \(8\)-regular graph,
and let \(\B\) be a balanced \(4\)-tracking decomposition of~\(G\) with no cycles.
We say that \(\B\) is \emph{complete} if the following hold.
\begin{enumerate}[(i)]
\item	\(\prehang(v,\B) \geq 1\), for every \(v\in V(G)\);
\item	if \(T\in\B\) contains a triangle, 
	then \(T\) or \(T^-\) is contained in 
	an exceptional pair of~\(\B\); and
\item	if $X$ is an exceptional pair of \(\B\) and $P$ is an element of $\B$
	that contains the central vertex of $X$ and a hanging edge at a connection vertex of $X$,
	then the central vertex of \(X\) is an end-vertex of $P$.
\end{enumerate}
\end{definition}

Now, we prove that the decomposition given by Lemma~\ref{lemma:path+exceptional} induces a complete decomposition.

\begin{lemma}\label{lemma:complete-decomposition}
	If \(G\) is an \(8\)-regular graph, 
	then \(G\) admits a complete decomposition.
\end{lemma}

\begin{proof}
	Let \(G\) be an \(8\)-regular graph.
	By Petersen's Theorem~\cite{Pe1891}, 
	\(G\) admits a \(2\)-factorization, 
	say \(\{F_1,F_2,F_3,F_4\}\).
	Thus, \(F = F_1 + F_2\) and \(H = F_3+F_4\) are \(4\)-factors of~\(G\).
	By Lemma~\ref{lemma:4-reg->2-path}, 
	\(F\) admits a balanced \(P_2\)-decomposition \(\D\).
	By Lemma~\ref{lemma:good-orientation}, 
	there exists a good orientation \(O\) of the edges of \(H\).
	By Lemma~\ref{lemma:path+exceptional},
	\(G\) admits an exceptional decomposition~\(\B\) 
	into \((\D,O)\)-extensions.
	In what follows, we prove that \(\B\) satisfies
	each item of Definition~\ref{def:complete}.
	By Fact~\ref{fact:balanced+prehang}, 
	\(\B\) is balanced and \(\prehang(v,\B) = 2\geq 1\)
	for every vertex~\(v\) of~\(G\).
	This proves item~(i) of 
	Definition~\ref{def:complete}.
	Since \(\B\) is an exceptional decomposition~\(\B\) has no cycles, 
	and if \(T\in\B\) contains a triangle,
	then \(T\) is contained in exactly one 
	exceptional extension,
	which implies that \(T\) is contained in exactly one 
	exceptional pair of \(\B\).
	This proves item~(ii) of Definition~\ref{def:complete}.
	
	Now, suppose \(X = \{T_1,T_2\}\) is an exceptional pair 
	of~\(\B\), and \(P\in \B\) is an element that contains 
	a hanging edge \(xy\) 
	at a connection vertex \(x\) of \(X\).
	Suppose that~\(P\) contains the center \(c\) of \(X\).
	Note that there are two hanging edges at \(c\) contained in \(E(T_1)\cup E(T_2)\).
	By the definition of \(X\), 
	we have \(xc\in E(T_1)\cup E(T_2)\).
	Therefore, \(P\) contains a path \(yxzc\), 
	for some vertex \(z\) of~\(G\).
	If~\(c\) is not an end-vertex of \(P\),
	then there is another vertex~\(z'\) such that 
	\(P\) contains the path \(yxzcz'\).
	Thus, \(P\) is exactly the tracking \(yxzcz'\),
	and \(cz\) is a hanging edge of~\(c\).
	Therefore there are three hanging edges at \(c\),
	a contradiction to \(\prehang(c,\B) = 2\).
	This proves item (iii) of Definition~\ref{def:complete}.
\end{proof}

Suppose \(\B\) is a complete \(4\)-tracking decomposition.
Let \(T = \{T_1,T_2\}\) be an exceptional pair, where \(T_1 = a_1b_1c_1d_1e_1\) is a path.
We say that the edge \(b_1d_1\) is the \emph{pivotal} edge of~\(T_1\).
Note that the pivotal edge of \(T_1\) is an edge of \(T_2\).
Therefore, \(T_1\) is contained in at most one exceptional pair.
Moreover, if \(w\) is the center of \(T\), then \(d_{T_1+T_2}(w) = 5\),
hence~\(w\) is not the center of any other exceptional pair.

Now we are able to prove our main theorem.
Recall that \(\tau'(\B)\) is the number of trackings of \(\B\) that induce paths.
\begin{theorem}
	If $G$ is an 8-regular graph, 
	then $G$ admits a balanced \(4\)-tracking decomposition.
\end{theorem}

\begin{proof}
Let \(G\) be an \(8\)-regular graph.
By Lemma~\ref{lemma:complete-decomposition},
\(G\) admits a complete decomposition.
Thus, let \(\B\) be a complete decomposition 
that maximizes \(\tau'(\B)\).
We claim that~\(\B\) is a path tracking decomposition.
Suppose, for a contradiction, that \(\B\) contains an element \(T\)
that contains a triangle.
By item (ii) of Definition~\ref{def:complete},
\(T\) is contained in an exceptional pair of \(\B\), 
say \(X=\{T_1,T_2\}\) of \(\B\).
Suppose \(T_1 = a_1b_1c_1d_1e_1\) and \(T_2 = a_2b_2c_2d_2e_2\),
where \(e_2 = b_2 = b_1\) and \(d_1 = d_2\).
For ease of notation, let \(b = b_1 = b_2\) and \(d = d_1 = d_2\).
By item~(i) of Definition~\ref{def:complete},
\(\prehang(c_1,\B) \geq 1\) and \(\prehang(c_2,\B) \geq 1\).
Thus, there is an element \(P_1\) containing a hanging edge at~\(c_1\),
and an element~\(P_2\) containing a hanging edge at~\(c_2\).
We claim that at least one between \(P_1\) and \(P_2\)
does not contain \(b\).
Indeed, suppose~\(P_1\) and \(P_2\) contain \(b\).
By item (iii) of Definition~\ref{def:complete},
\(b\) is an end-vertex of~\(P_1\) and \(P_2\).
But \(b\) is an end-vertex of \(T_2\).
Therefore, \(\B(b) \geq 3\), a contradiction to~\(\B\) being a balanced
decomposition.

Thus, we may assume that \(P_i\)
does not contain \(b\), and put \(j=3-i\).
Without loss of generality, let \(P_i = a_3b_3c_3d_3e_3\), where \(b_3 = c_i\)
and \(a_3b_3\) is a hanging edge at \(c_i\) 
(otherwise we have \(P_i^- = a_3b_3c_3d_3e_3\), where \(b_3 = c_i\)
and \(a_3b_3\) is a hanging edge at \(c_i\)).
Put \(P' = bb_3c_3d_3e_3\), and note that since \(P_i\) 
does not contain~\(b\), 
\(P'\) contains a triangle only if \(P_i\) 
contains the triangle \(b_3c_3d_3e_3\).
Now we show how to decompose the subgraph of \(G\)
induced by \(E(T_1) + E(T_2) -c_ib + c_ia_3\) into paths of length \(4\).
We divide the proof into two cases, whether \(a_2 = e_1\) or \(a_2\neq e_1\).
If \(a_2=e_1\),
then since~\(\B\) is balanced, we have \(a_2\neq a_3\).
Put \(T_1' = a_3c_idba_2\) and \(T_2' = a_1bc_jde_1\) (see Figure~\ref{fig:fixing-exceptional}). 
Now, suppose \(a_2\neq e_1\).
Since \(\B\) is balanced, we have \(a_3\neq a_1\) or \(a_3\neq a_2\).
Say \(a_3\neq a_k\), where \(k\in\{1,2\}\), and put \(l=3-k\).
Note that, since \(T_1\) is a path, we have \(a_1\neq e_1\), hence \(a_k,a_l\neq e_1\).
We put \(T_1' = a_3c_idba_k\) and \(T_2' = a_lbc_jde_1\).

\begin{figure}[h]\centering
	\centering
	\begin{subfigure}[b]{.5\linewidth}
	\centering
	\scalebox{.9}{\input{WTF-labeled-new}}
	\end{subfigure}%
	\begin{subfigure}[b]{.5\linewidth}
	\centering
	\scalebox{.9}{\input{WTF-labeled-new2}}
	\end{subfigure}%
	\caption{Decomposing $T_1+T_2+P_i$ when $a_2 = e_1$.}\vspace{-.5cm}
	\label{fig:fixing-exceptional}	
\end{figure}
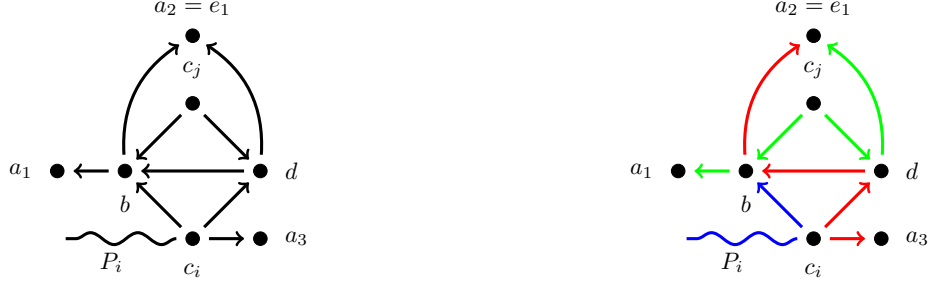

It is clear that in the above two cases \(T_1'\) and \(T_2'\) are paths of length \(4\).
Let \(\B' = \B - P_i + P' - T_1 + T_1' - T_2 + T_2'\)
(note that we supposed that \(P_1,T_1,T_2\in\B\), otherwise, we use \(P_i^-,T_1^-,T_2^-\), 
conveniently).
We have \(\tau'(\B') \geq \tau'(\B) + 1\).
Now, we verify that \(\B'\) is a complete decomposition.
Clearly, \(\B'\) is balanced, since \(\B'(v) = 2\) for every vertex \(v\) of \(G\).

\smallskip
\noindent\textbf{(i)}
It is clear that the only edge that is a hanging edge in \(\B\) but not in \(\B'\) 
is \(db\).
But \(d_1e_1\) is a hanging edge at~\(d\),
thus \(\prehang(d,\B')\geq 1\).
Moreover, if \(v\neq d\), and \(e\) is a hanging edge at \(v\) in \(\B\),
then \(e\) is a hanging edge at~\(v\) in~\(\B'\),
hence \(\prehang(v,\B') \geq \prehang(v,\B) \geq 1\)
for every vertex \(v\neq d\).

\smallskip
\noindent\textbf{(ii)}
Let \(Q\) be an element of \(\B'\) that contains a triangle.
If \(Q\notin \{P',T_1',T_2'\}\), 
then \(Q\) is an element of \(\B\cup\B^-\), 
hence, by item (ii) of Definition~\ref{def:complete}, 
\(Q\) is contained in at least one exceptional pair of \(\B\), say \(X_R = \{Q,W\}\)
\big(otherwise~\(Q^-\) is contained in an exceptional pair \(X_R = \{Q^-,R\}\)\big).
As noted before,~\(X_R\) is the only exceptional pair of \(\B\) containing \(W\).
Therefore, \(W\neq T_1\).
If \(W=P_i\), then \(\{Q,P'\}\) is an exceptional pair of \(\B'\).
Now, suppose \(Q\in \{P',T_1',T_2'\}\). 
Thus, \(Q=P'\), because~\(T_1'\) and~\(T_2'\) do not contain triangles.
As noted before, 
\(P'\) contains a triangle only if \(P_i\) contains the triangle \(b_3c_3d_3e_3\).
Note that \(\{P_i,T_j\}\) is not an exceptional pair of~\(\B\).
Thus, there is a path \(W\) in \(\B\) such that \(X_R=\{P_i,W\}\)
is an exceptional pair of \(\B\).
Therefore, \(\{P',W\}\) is an exceptional pair of~\(\B'\).

\smallskip	
\noindent\textbf{(iii)}
Let \(X' = \{Q,W\}\) be an exceptional pair of \(\B'\) with central vertex \(z\),
and let \(R\) be an element of \(\B'\) that contains a hanging edge at a connection vertex, 
say \(c\), of~\(X'\) and contains \(z\).
It is clear that \(z\neq c\).
Since \(\B\) is complete, if \(Q,W,R\notin\{P',T_1',T_2'\}\), then \(z\) is an end-vertex of~\(R\).
Now, we claim that \(Q,W\neq T_1',T_2'\).
Indeed, since the pivotal edge of \(T_1'\) is in~\(E(T_2')\), if \(Q = T_1'\), then \(W = T_2'\),
but \(T_1'\) and \(T_2'\) are paths, hence \(\{T_1',T_2'\}\) is not an exceptional pair of \(\B'\).
Analogously, we have \(Q \neq T_2'\).
Therefore, \(P'\) is the only element in \(\{P',T_1',T_2'\}\) that is contained in an exceptional pair of \(\B'\).
In what follows we first study the cases where \(R\in\{P',T_1',T_2'\}\),
and then the cases where \(R\notin\{P',T_1',T_2'\}\).

Suppose \(R = P'\) and recall that \(P' = bb_3c_3d_3e_3\).
In this case, we have \(Q,W\in\B\) and \(c\in\{b_3,d_3\}\).
Thus, \(P_i\) contains (in \(\B\)) a hanging edge at \(c\).
Since \(\B\) is complete, either \(z\notin V(P_i)\) or \(z\) is an end-vertex of~\(P_i\).
The only vertex in \(V(P')-V(P_i)\)
is \(b\), which is the central vertex of \(\{T_1,T_2\}\).
Thus, if \(z\notin V(P_i)\), then \(z = b\), 
hence \(z\) is the central vertex of at least two exceptional pairs of~\(\B\), a contradiction.
Therefore, we may assume that~\(z\) is an end-vertex of \(P_i\), i.e., \(z \in \{a_3,e_3\}\).
Suppose \(z=a_3\).
By the construction of~\(P'\), we have \(a_3\in V(P')\)
if and only if \(P_i\) contains the triangle \(a_3b_3c_3d_3\).
In this case, \(z = a_3 = d_3\), hence \(c\neq d_3\) because \(z\neq c\).
Thus \(c = b_3\), 
hence \(b_3a_3\in E(P_i)\) and \(cz\in E(Q)\cup E(W)\) are parallel edges, a contradiction to \(G\) being simple.
Thus, \(z=e_3\) is an end-vertex of \(P'\).

Now, we study the cases where \(R = T_1'\) or \hbox{\(R=T_2'\).}
First, note that since \(\big(E(Q)\cup E(W)\big)\cap\big(E(T_1)\cup E(T_2)\big) = \emptyset\),
\(d_{E(Q)\cup E(W)}(z) = d_{E(T_1)\cup E(T_2)}(b) = 5\), and \(d_{E(T_1)\cup E(T_2)}(d) = 4\),
we have \(z\neq b,d\), otherwise \(d(z) \geq 9\), a contradiction to \(G\) being \(8\)-regular.
Suppose that \(R = T_1'\).
In this case, we have \(R = a_3c_idba_k\), where \(k\in\{1,2\}\),
hence \(c=c_i\) or \(c=b\).
Since \(z\neq b,d\), if \(c = c_i\), then \(z\in \{a_3,a_k\}\).
If \(c=b\) and \(z=c_i\), 
then \(bc_i\in E(T_2')\) and \(cz\in E(Q)\cup E(W)\) are parallel edges, 
a contradiction.
Thus, if \(c=b\), then \(z\in \{a_3,a_k\}\).
In both cases we have \(z\in \{a_3,a_k\}\), which are precisely the end-vertices of \(R\).
Suppose that \(R = T_2'\).
In this case, we have \(R = a_lbc_jde_1\), where \(l\in\{1,2\}\).
Thus, \(c=b\) or \(c=d\).
In any of this cases, \(z\neq c_j\), 
otherwise \(cc_j\in E(T_2')\) and \(cz\in E(Q)\cup E(W)\) would be parallel edges, 
a contradiction.
Since \(z\neq b,d\), we have \(z\in \{a_l,e_1\}\), which are precisely the end-vertices of \(R\).

As noted before, at least one between \(Q,W,R\) must be in \(\{P',T_1',T_2'\}\).
We already studied the cases there \(R\in\{P',T_1',T_2'\}\).
Now it remains to prove the case where \(R\notin\{P',T_1',T_2'\}\).
Also, we know that \(T_1',T_2'\notin \{Q,W\}\).
Thus \(P'\in \{Q,W\}\) and \(R \neq T_1',T_2'\).
Suppose, without of generality, that \(Q = P'\).
In this case, we have \(W,R\notin\{P',T_1',T_2'\}\).
Since \(\{P',W\}\) is an exceptional pair in \(\B'\),
\(\{P_i,W\}\) is an exceptional pair in \(\B\),
because the middle edges, \(b_3c_3\) and \(c_3d_3\), of \(P_i\) and \(P'\) are the same.
Moreover, \(z\) is the central vertex of \(\{P_i,W\}\) 
and \(c\) is a connection vertex of \(\{P_i,W\}\).
Thus, since \(\B\) is complete and \(R\) contains \(z\), \(z\) is an end-vertex of \(R\).
Therefore, \(\B'\) satisfies item (iii) of Definition~\ref{def:complete}.

We conclude that \(\B'\) is a complete decomposition such that \(\tau'(\B') > \tau'(\B)\),
a contradiction to the maximality of \(\tau'(\B)\).
\end{proof}

\begin{corollary}
	If \(G\) is an \(8\)-regular graph, then \(G\) admits a balanced \(P_4\)-decomposition.
\end{corollary}

\section{Concluding remarks}

In this paper we prove Conjecture~\ref{conj:koulonc} for paths of length \(4\).
This result improves the previous result~\cite{KoLo99} that, for paths of length \(4\),
states that triangle-free \(8\)-regular graphs admit balanced \(P_4\)-decompositions.
We believe that the technique presented here can be modified to improve the girth condition
for \(\ell >4\), or to prove Conjecture~\ref{conj:haggkvist} for trees of diameter \(4\).

\bibliography{bibliografia}

\end{document}

%% file: extension1.tex
\begin{tikzpicture}[scale = 1]


\node (0) [black vertex] at (-1,2) {};
\node (1) [black vertex] at (-1,1) {};
\node (2) [black vertex] at (0,0) {};
\node (3) [black vertex] at (1,1) {};
\node (4) [black vertex] at (1,2) {};

\draw [line width=1.3pt,color=red,->] (2) -- (1);
\draw [line width=1.3pt,color=red,->] (2) -- (3);
\draw [line width=1.3pt,color=black,->] (1) -- (0);
\draw [line width=1.3pt,color=black,->] (3) -- (4);

\end{tikzpicture}

%% file: extension2.tex
\begin{tikzpicture}[scale = 1]


\node (0) [black vertex] at (-1,2) {};
\node (1) [black vertex] at (-1,1) {};
\node (2) [black vertex] at (0,0) {};
\node (3) [black vertex] at (1,1) {};
\node (4) [black vertex] at (-1,1) {};

\draw [line width=1.3pt,color=red,->] (2) -- (1);
\draw [line width=1.3pt,color=red,->] (2) -- (3);
\draw [line width=1.3pt,color=black,->] (1) -- (0);
\draw [line width=1.3pt,color=black,->] (3) -- (4);

\end{tikzpicture}

%% file: extension3.tex
\begin{tikzpicture}[scale = 1]


\node (0) [black vertex] at (0,2) {};
\node (1) [black vertex] at (-1,1) {};
\node (2) [black vertex] at (0,0) {};
\node (3) [black vertex] at (1,1) {};
\node (4) [black vertex] at (0,2) {};

\draw [line width=1.3pt,color=red,->] (2) -- (1);
\draw [line width=1.3pt,color=red,->] (2) -- (3);
\draw [line width=1.3pt,color=black,->] (1) -- (0);
\draw [line width=1.3pt,color=black,->] (3) -- (4);

\end{tikzpicture}

%% file: trapped-P2.tex
\begin{tikzpicture}[scale = 1]
\def\r3{1.73205}
\def\y{1.4142}

\node (a) [black vertex] at (0.5-\r3*0.5,\r3*0.5+0.5) {};
\node (b) [black vertex] at (1,\r3) {};
\node (c) [black vertex] at (1.5+\r3*0.5, \r3*0.5+0.5) {};
\node (d) [black vertex] at (2,0) {};
\node (e) [black vertex] at (1,-1) {};
\node (f) [black vertex] at (0,0) {};

\draw [line width=1.3pt,color=red,->] (a) -- (b);
\draw [line width=1.3pt,color=red,->] (a) -- (f);
\draw [line width=1.3pt,color=red,->] (c) -- (b);
\draw [line width=1.3pt,color=red,->] (c) -- (d);
\draw [line width=1.3pt,color=red,->] (e) -- (d);
\draw [line width=1.3pt,color=red,->] (e) -- (f);

\draw [line width=1.3pt,color=black, -] (f) -- (d);
\draw [line width=1.3pt,color=black, -] (d) -- (b);
		
\end{tikzpicture}

%% file: trapped-triangle.tex
\begin{tikzpicture}[scale = 1]
\def\r3{1.73205}
\def\y{1.4142}

\node (a) [black vertex] at (0.5-\r3*0.5,\r3*0.5+0.5) {};
\node (b) [black vertex] at (1,\r3) {};
\node (c) [black vertex] at (1.5+\r3*0.5, \r3*0.5+0.5) {};
\node (d) [black vertex] at (2,0) {};
\node (e) [black vertex] at (1,-1) {};
\node (f) [black vertex] at (0,0) {};

\draw [line width=1.3pt,color=red,->] (a) -- (b);
\draw [line width=1.3pt,color=red,->] (a) -- (f);
\draw [line width=1.3pt,color=red,->] (c) -- (b);
\draw [line width=1.3pt,color=red,->] (c) -- (d);
\draw [line width=1.3pt,color=red,->] (e) -- (d);
\draw [line width=1.3pt,color=red,->] (e) -- (f);

\draw [line width=1.3pt,color=black, -] (f) -- (d);
\draw [line width=1.3pt,color=black, -] (d) -- (b);
\draw [line width=1.3pt,color=black, -] (b) -- (f);
		
\end{tikzpicture}

%% file: quasi-trapped-triangle-first.tex
\begin{tikzpicture}[scale = 1]
\def\r3{1.73205}
\def\y{1.4142}

\node (b) [black vertex] at (1,\r3) {};
\node (c) [black vertex] at (1.5+\r3*0.5, \r3*0.5+0.5) {};
\node (d) [black vertex] at (2,0) {};
\node (e) [black vertex] at (1,-1) {};
\node (f) [black vertex] at (0,0) {};

\draw [line width=1.3pt,color=red,->] (c) -- (b);
\draw [line width=1.3pt,color=red,->] (c) -- (d);
\draw [line width=1.3pt,color=red,->] (e) -- (d);
\draw [line width=1.3pt,color=red,->] (e) -- (f);

\draw [line width=1.3pt,color=black, -] (f) -- (d);
\draw [line width=1.3pt,color=black, -] (d) -- (b);
\draw [line width=1.3pt,color=black, -] (b) -- (f);

\phantom{
	\node (a) [black vertex] at (0.5-\r3*0.5,\r3*0.5+0.5) {};
}
		
\end{tikzpicture}

%% file: trapped-K4.tex
\begin{tikzpicture}[scale = 0.7]


\node (alpha) [black vertex] at (0,0) {};
\node (a) [black vertex] at (2,0) {};
\node (beta) [black vertex] at (4,0) {};
\node (b) [black vertex] at (4,-2) {};

\node (gamma) [black vertex] at (4,-4) {};
\node (c) [black vertex] at (2,-4) {};
\node (delta) [black vertex] at (0,-4) {};

\node (d) [black vertex] at (0,-2) {};

\draw [line width=1.3pt,color=red,->] (alpha) -- (a);
\draw [line width=1.3pt,color=red,->] (alpha) -- (d);
\draw [line width=1.3pt,color=red,->] (beta) -- (a);
\draw [line width=1.3pt,color=red,->] (beta) -- (b);
\draw [line width=1.3pt,color=red,->] (gamma) -- (b);
\draw [line width=1.3pt,color=red,->] (gamma) -- (c);
\draw [line width=1.3pt,color=red,->] (delta) -- (c);
\draw [line width=1.3pt,color=red,->] (delta) -- (d);

\draw [line width=1.3pt,color=black] (a) -- (c);
\draw [line width=1.3pt,color=black] (b) -- (a);
\draw [line width=1.3pt,color=black] (b) -- (d);
\draw [line width=1.3pt,color=black] (c) -- (b);
\draw [line width=1.3pt,color=black] (c) -- (d);
\draw [line width=1.3pt,color=black] (d) -- (a);

\end{tikzpicture}

%% file: quasi-eulerian.tex
\begin{tikzpicture}[scale = 1,rotate=-60]
\def\r3{1.73205}
\def\y{1.4142}

\node (b) [black vertex] at (1,\r3) {};
\node (c) [black vertex] at (1.5+\r3*0.5, \r3*0.5+0.5) {};
\node (d) [black vertex] at (2,0) {};
\node (e) [black vertex] at (1,-1) {};
\node (f) [black vertex] at (0,0) {};

\draw [line width=1.3pt,color=red,->] (c) -- (b);
\draw [line width=1.3pt,color=red,->] (c) -- (d);
\draw [line width=1.3pt,color=red,->] (e) -- (d);
\draw [line width=1.3pt,color=red,->] (e) -- (f);

\draw [line width=1.3pt,color=black, ->] (f) -- (d);
\draw [line width=1.3pt,color=black, ->] (d) -- (b);
\draw [line width=1.3pt,color=black, ->] (b) -- (f);
		
\end{tikzpicture}

%% file: quasi-centered1.tex
\begin{tikzpicture}[scale = 1,rotate=-60]
\def\r3{1.73205}
\def\y{1.4142}

\node (b) [black vertex] at (1,\r3) {};
\node (c) [black vertex] at (1.5+\r3*0.5, \r3*0.5+0.5) {};
\node (d) [black vertex] at (2,0) {};
\node (e) [black vertex] at (1,-1) {};
\node (f) [black vertex] at (0,0) {};

\draw [line width=1.3pt,color=red,->] (c) -- (b);
\draw [line width=1.3pt,color=red,->] (c) -- (d);
\draw [line width=1.3pt,color=red,->] (e) -- (d);
\draw [line width=1.3pt,color=red,->] (e) -- (f);

\draw [line width=1.3pt,color=black, ->] (d) -- (f);
\draw [line width=1.3pt,color=black, ->] (d) -- (b);
\draw [line width=1.3pt,color=black, ->] (b) -- (f);
		
\end{tikzpicture}

%% file: quasi-centered2.tex
\begin{tikzpicture}[scale = 1,rotate=-60]
\def\r3{1.73205}
\def\y{1.4142}

\node (b) [black vertex] at (1,\r3) {};
\node (c) [black vertex] at (1.5+\r3*0.5, \r3*0.5+0.5) {};
\node (d) [black vertex] at (2,0) {};
\node (e) [black vertex] at (1,-1) {};
\node (f) [black vertex] at (0,0) {};

\draw [line width=1.3pt,color=red,->] (c) -- (b);
\draw [line width=1.3pt,color=red,->] (c) -- (d);
\draw [line width=1.3pt,color=red,->] (e) -- (d);
\draw [line width=1.3pt,color=red,->] (e) -- (f);

\draw [line width=1.3pt,color=black, ->] (f) -- (d);
\draw [line width=1.3pt,color=black, ->] (b) -- (d);
\draw [line width=1.3pt,color=black, ->] (b) -- (f);
		
\end{tikzpicture}

%% file: quasi-trapped-triangle-new2.tex
\begin{tikzpicture}[scale=0.8, ,dots/.style={circle,draw=black,fill=black,thick,
inner sep=0pt,minimum size=0.6mm}]

\node (a0) [black vertex, label=below:$a_0$] at (0,0) {};
\node (a1) [black vertex, label=below:$a_1$] at (2,0) {};
\node (a2) [black vertex, label=below:$a_2$] at (4,0) {};
\node (ap2) [black vertex, label=below:$a_{k-1}$] at (10,0) {};
\node (ap1) [black vertex, label=below:$a_{k}$] at (12,0) {};
\node (ap) [black vertex, label=below:$a_{k+1}$] at (14,0) {};

\node (a4) [] at (6,0) {};
\node (ap3) [] at (8,0) {};

\node (alpha0) [black vertex] at (1,-1) {};
\node (alpha1) [black vertex] at (3,-1) {};
\node (alphap2) [black vertex] at (11,-1) {};
\node (alphap1) [black vertex] at (13,-1) {};

\draw [line width=1.3pt,color=red,->] (alpha0) -- (a0);
\draw [line width=1.3pt,color=red,->] (alpha0) -- (a1);
\draw [line width=1.3pt,color=red,->] (alpha1) -- (a1);
\draw [line width=1.3pt,color=red,->] (alpha1) -- (a2);

\draw [line width=1.3pt,color=red,->] (alphap2) -- (ap2);
\draw [line width=1.3pt,color=red,->] (alphap2) -- (ap1);
\draw [line width=1.3pt,color=red,->] (alphap1) -- (ap1);
\draw [line width=1.3pt,color=red,->] (alphap1) -- (ap);

\draw [line width=1.3pt,color=black,->] (a0) -- (a1);

\draw [line width=1.3pt,color=black,->] (a1) -- (a2);


\draw [line width=1.3pt,color=black,->] (ap2) -- (ap1);

\draw [line width=1.3pt,color=black,->] (ap1) -- (ap);

\def\xx{6.7}
\node (t) at (\xx,0) [dots] {};
\node (tt) at (\xx+0.3,0) [dots] {};
\node (ttt) at (\xx+0.6,0) [dots] {};

\draw (a2) edge [line width=1.3pt,color=black, 
			bend right=45,->] (a0);
\draw (ap) edge [line width=1.3pt,color=black, 
			bend right=45,->] (ap2);
\draw (ap1) edge [dashed, line width=1.3pt,color=black, 
			bend right=45,->] (ap3);
\draw (a4) edge [dashed, line width=1.3pt,color=black, 
			bend right=45,->] (a1);

\end{tikzpicture}

%% file: trapped-K4-orientation1.tex
\begin{tikzpicture}[scale = 0.7]


\node (alpha) [black vertex] at (-2,2) {};
\node (a) [black vertex, label={[xshift=0.5cm]above:$x_3$}] at (0,2) {};
\node (aout) [] at (0,3) {};

\node (beta) [black vertex] at (2,2) {};
\node (b) [black vertex, label={[yshift=-0.5cm]right:$x_0$}] at (2,0) {};
\node (bout) [] at (3,0) {};

\node (gamma) [black vertex] at (2,-2) {};
\node (c) [black vertex, label={[xshift=-0.5cm]below:$x_1$}] at (0,-2) {};
\node (cout) [] at (0,-3) {};

\node (delta) [black vertex] at (-2,-2) {};
\node (d) [black vertex, label={[yshift=0.5cm]left:$x_2$}] at (-2,0) {};
\node (dout) [] at (-3,0) {};

\draw [line width=1.3pt,color=red,->] (alpha) -- (a);
\draw [line width=1.3pt,color=red,->] (alpha) -- (d);
\draw [line width=1.3pt,color=red,->] (beta) -- (a);
\draw [line width=1.3pt,color=red,->] (beta) -- (b);
\draw [line width=1.3pt,color=red,->] (gamma) -- (b);
\draw [line width=1.3pt,color=red,->] (gamma) -- (c);
\draw [line width=1.3pt,color=red,->] (delta) -- (c);
\draw [line width=1.3pt,color=red,->] (delta) -- (d);

\draw [line width=1.3pt,color=black] (a) -- (c);
\draw [line width=1.3pt,color=black] (b) -- (a);
\draw [line width=1.3pt,color=black] (b) -- (d);
\draw [line width=1.3pt,color=black] (c) -- (b);
\draw [line width=1.3pt,color=black] (c) -- (d);
\draw [line width=1.3pt,color=black] (d) -- (a);

\draw [line width=1.3pt,color=black] (cout) -- (c);
\draw [line width=1.3pt,color=black] (a) -- (aout);
\draw [line width=1.3pt,color=black] (bout) -- (b);
\draw [line width=1.3pt,color=black] (d) -- (dout);

\end{tikzpicture}

%% file: trapped-K4-orientation3.tex
\begin{tikzpicture}[scale = 0.7]


\node (alpha) [black vertex] at (-2,2) {};
\node (a) [black vertex, label={[xshift=0.5cm]above:$x_3$}] at (0,2) {};
\node (aout) [] at (0,3) {};

\node (beta) [black vertex] at (2,2) {};
\node (b) [black vertex, label={[yshift=-0.5cm]right:$x_0$}] at (2,0) {};
\node (bout) [] at (3,0) {};

\node (gamma) [black vertex] at (2,-2) {};
\node (c) [black vertex, label={[xshift=-0.5cm]below:$x_1$}] at (0,-2) {};
\node (cout) [] at (0,-3) {};

\node (delta) [black vertex] at (-2,-2) {};
\node (d) [black vertex, label={[yshift=0.5cm]left:$x_2$}] at (-2,0) {};
\node (dout) [] at (-3,0) {};

\draw [line width=1.3pt,color=red,->] (alpha) -- (a);
\draw [line width=1.3pt,color=red,->] (alpha) -- (d);
\draw [line width=1.3pt,color=red,->] (beta) -- (a);
\draw [line width=1.3pt,color=red,->] (beta) -- (b);
\draw [line width=1.3pt,color=red,->] (gamma) -- (b);
\draw [line width=1.3pt,color=red,->] (gamma) -- (c);
\draw [line width=1.3pt,color=red,->] (delta) -- (c);
\draw [line width=1.3pt,color=red,->] (delta) -- (d);

\draw [line width=1.3pt,color=black,->] (c) -- (a);
\draw [line width=1.3pt,color=black,->] (b) -- (d);

\draw [line width=1.3pt,color=black,->] (cout) -- (c);
\draw [line width=1.3pt,color=black,->] (a) -- (aout);
\draw [line width=1.3pt,color=black,->] (bout) -- (b);
\draw [line width=1.3pt,color=black,->] (d) -- (dout);

\end{tikzpicture}

%% file: trapped-K4-orientation4.tex
\begin{tikzpicture}[scale = 0.7]


\node (alpha) [black vertex] at (-2,2) {};
\node (a) [black vertex, label={[xshift=0.5cm]above:$x_3$}] at (0,2) {};
\node (aout) [] at (0,3) {};

\node (beta) [black vertex] at (2,2) {};
\node (b) [black vertex, label={[yshift=-0.5cm]right:$x_0$}] at (2,0) {};
\node (bout) [] at (3,0) {};

\node (gamma) [black vertex] at (2,-2) {};
\node (c) [black vertex, label={[xshift=-0.5cm]below:$x_1$}] at (0,-2) {};
\node (cout) [] at (0,-3) {};

\node (delta) [black vertex] at (-2,-2) {};
\node (d) [black vertex, label={[yshift=0.5cm]left:$x_2$}] at (-2,0) {};
\node (dout) [] at (-3,0) {};

\draw [line width=1.3pt,color=red,->] (alpha) -- (a);
\draw [line width=1.3pt,color=red,->] (alpha) -- (d);
\draw [line width=1.3pt,color=red,->] (beta) -- (a);
\draw [line width=1.3pt,color=red,->] (beta) -- (b);
\draw [line width=1.3pt,color=red,->] (gamma) -- (b);
\draw [line width=1.3pt,color=red,->] (gamma) -- (c);
\draw [line width=1.3pt,color=red,->] (delta) -- (c);
\draw [line width=1.3pt,color=red,->] (delta) -- (d);

\draw [line width=1.3pt,color=black,->] (a) -- (c);
\draw [line width=1.3pt,color=black,->] (b) -- (a);
\draw [line width=1.3pt,color=black,->] (b) -- (d);
\draw [line width=1.3pt,color=black,->] (c) -- (b);
\draw [line width=1.3pt,color=black,->] (c) -- (d);
\draw [line width=1.3pt,color=black,->] (d) -- (a);

\draw [line width=1.3pt,color=black,->] (cout) -- (c);
\draw [line width=1.3pt,color=black,->] (a) -- (aout);
\draw [line width=1.3pt,color=black,->] (bout) -- (b);
\draw [line width=1.3pt,color=black,->] (d) -- (dout);

\end{tikzpicture}

%% file: exceptional.tex
\begin{tikzpicture}[scale = 1]


\node (0) [black vertex] at (-2,0) {};
\node (1) [black vertex] at (-1,1) {};
\node (2) [black vertex] at (0,0) {};
\node (3) [black vertex] at (1,1) {};
\node (4) [black vertex] at (-1,1) {};

\node (0') [black vertex] at (-2,2) {};
\node (2') [black vertex] at (0,2) {};
\node (4') [black vertex] at (2,1) {};

\draw [line width=1.3pt,color=red,->] (2) -- (1);
\draw [line width=1.3pt,color=red,->] (2) -- (3);
\draw [line width=1.3pt,color=black,->] (1) -- (0);
\draw [line width=1.3pt,color=black,->] (3) -- (4);

\draw [line width=1.3pt,color=red,->] (2') -- (1);
\draw [line width=1.3pt,color=red,->] (2') -- (3);
\draw [line width=1.3pt,color=black,->] (1) -- (0');
\draw [line width=1.3pt,color=black,->] (3) -- (4');
\end{tikzpicture}

%% file: exceptional2.tex
\begin{tikzpicture}[scale = 1]


\node (0) [black vertex] at (-2,1) {};
\node (1) [black vertex] at (-1,1) {};
\node (2) [black vertex] at (0,0) {};
\node (3) [black vertex] at (1,1) {};
\node (4) [black vertex] at (-1,1) {};

\node (0') [black vertex] at (0,3) {};
\node (2') [black vertex] at (0,2) {};
\node (4') [] at (2,1) {};

\draw [line width=1.3pt,color=red,->] (2) -- (1);
\draw [line width=1.3pt,color=red,->] (2) -- (3);
\draw [line width=1.3pt,color=black,->] (1) -- (0);
\draw [line width=1.3pt,color=black,->] (3) -- (4);

\draw [line width=1.3pt,color=red,->] (2') -- (1);
\draw [line width=1.3pt,color=red,->] (2') -- (3);
\draw [line width=1.3pt,color=black,->] (1) edge [bend left] (0');
\draw [line width=1.3pt,color=black,->] (3) edge [bend right] (0');
	
\end{tikzpicture}

%% file: WTF-labeled-new.tex
\begin{tikzpicture}[scale = 1]


\node (v) [black vertex, label=left:$a_1$] at (-2,1) {};
\node (a) [black vertex, label=below:$b$] at (-1,1) {};
\node (beta) [black vertex, label=below:$c_i$] at (0,0) {};
\node (b) [black vertex, label=right:$d$] at (1,1) {};

\node (alpha) [black vertex, label=above:$c_j$] at (0,2) {};

\node () [] at (2,1) {};

\node (y) [black vertex, label=right:$a_3$] at (1,0) {};

\node (top) [black vertex] at (0,3) {};
\node () [] at (0,3.4) {$a_2= e_1$};

\draw [line width=1.3pt,color=black,->] (beta) -- (a);
\draw [line width=1.3pt,color=black,->] (beta) -- (b);
\draw [line width=1.3pt,color=black,->] (a) -- (v);
\draw [line width=1.3pt,color=black,->] (b) -- (a);

\draw [line width=1.3pt,color=black,->] (alpha) -- (a);
\draw [line width=1.3pt,color=black,->] (alpha) -- (b);
\draw [line width=1.3pt,color=black,->] (a) to [bend left] (top);
\draw [line width=1.3pt,color=black,->] (b) to [bend right] (top);
\draw [line width=1.3pt,color=black,->] (beta) -- (y);

 \node (l1) [] at (-2,0){};
 \node (l2) [] at (-2,3){};
  
 \draw[E edge][decorate,decoration={snake,segment length=7mm}] 
				(beta) -- (l1);

\node () [] at (-1.2, -0.35) {$P_i$};

\end{tikzpicture}

%% file: WTF-labeled-new2.tex
\begin{tikzpicture}[scale = 1]


\node (v) [black vertex, label=left:$a_1$] at (-2,1) {};
\node (a) [black vertex, label=below:$b$] at (-1,1) {};
\node (beta) [black vertex, label=below:$c_i$] at (0,0) {};
\node (b) [black vertex, label=right:$d$] at (1,1) {};

\node (alpha) [black vertex, label=above:$c_j$] at (0,2) {};

\node () [] at (2,1) {};

\node (y) [black vertex, label=right:$a_3$] at (1,0) {};

\node (top) [black vertex] at (0,3) {};
\node () [] at (0,3.4) {$a_2= e_1$};

\draw [line width=1.3pt,color=black,->,color=blue] (beta) -- (a);
\draw [line width=1.3pt,color=black,->,color=red] (beta) -- (b);
\draw [line width=1.3pt,color=black,->,color=green] (a) -- (v);
\draw [line width=1.3pt,color=black,->,color=red] (b) -- (a);

\draw [line width=1.3pt,color=black,->,color=green] (alpha) -- (a);
\draw [line width=1.3pt,color=black,->,color=green] (alpha) -- (b);
\draw [line width=1.3pt,color=black,->,color=red] (a) to [bend left] (top);
\draw [line width=1.3pt,color=black,->,color=green] (b) to [bend right] (top);
\draw [line width=1.3pt,color=black,->,color=red] (beta) -- (y);

 \node (l1) [] at (-2,0){};
 \node (l2) [] at (-2,3){};
  
 \draw[E edge][decorate,decoration={snake,segment length=7mm},color=blue] 
				(beta) -- (l1);

\node () [] at (-1.2, -0.35) {$P_i$};

\end{tikzpicture}

%% file: 4-paths.bbl
\begin{thebibliography}{10}

\bibitem{BoMoOsWa15+thomassen}
F\'abio Botler, Guilherme~O. Mota, Marcio T.~I. Oshiro, and Yoshiko
  Wakabayashi.
\newblock Decomposing highly edge-connected graphs into paths of any given
  length.
\newblock submitted, 2015.

\bibitem{BoMoOsWa15+reg}
F\'abio Botler, Guilherme~O. Mota, Marcio T.~I. Oshiro, and Yoshiko
  Wakabayashi.
\newblock Decomposing regular graphs with prescribed girth into paths.
\newblock submitted, 2016.

\bibitem{EdHo06}
Michelle Edwards and Lea Howard.
\newblock A survey of graceful trees.
\newblock {\em Atl. Electron. J. Math.}, 1(1):5--30, 2006.

\bibitem{Er15+}
Joshua Erde.
\newblock Decomposing the cube into paths.
\newblock {\em Discrete Math.}, 336:41--45, 2014.

\bibitem{Fi90}
John~Frederick Fink.
\newblock On the decomposition of {$n$}-cubes into isomorphic trees.
\newblock {\em J. Graph Theory}, 14(4):405--411, 1990.

\bibitem{Golomb72}
Solomon~W. Golomb.
\newblock How to number a graph.
\newblock In {\em Graph theory and computing}, pages 23--37. Academic Press,
  New York, 1972.

\bibitem{Ha89}
Roland H{\"a}ggkvist.
\newblock Decompositions of complete bipartite graphs.
\newblock In {\em Surveys in combinatorics, 1989 ({N}orwich, 1989)}, volume 141
  of {\em London Math. Soc. Lecture Note Ser.}, pages 115--147. Cambridge Univ.
  Press, Cambridge, 1989.

\bibitem{JaTrTu91}
Michael~S. Jacobson, Miros{\l}aw Truszczy{\'n}ski, and Zsolt Tuza.
\newblock Decompositions of regular bipartite graphs.
\newblock {\em Discrete Math.}, 89(1):17--27, 1991.

\bibitem{JaKoWe13+}
Kyle~F. Jao, Alexandr~V. Kostochka, and Douglas~B. West.
\newblock Decomposition of {C}artesian products of regular graphs into
  isomorphic trees.
\newblock {\em J. Comb.}, 4(4):469--490, 2013.

\bibitem{KoLo99}
Mekkia Kouider and Zbigniew Lonc.
\newblock Path decompositions and perfect path double covers.
\newblock {\em Australas. J. Combin.}, 19:261--274, 1999.

\bibitem{Pe1891}
Julius Petersen.
\newblock Die {T}heorie der regul\"aren graphs.
\newblock {\em Acta Math.}, 15(1):193--220, 1891.

\bibitem{Ringel64}
G.~Ringel.
\newblock Problem n.25.
\newblock In {\em Theory of {G}raphs and its {A}pplications ({P}roc. {S}ympos.
  {S}molenice, 1963)}. Publ. House Czechoslovak Acad. Sci., Prague, 1964.

\bibitem{Rosa67}
A.~Rosa.
\newblock On certain valuations of the vertices of a graph.
\newblock In {\em Theory of {G}raphs ({I}nternat. {S}ympos., {R}ome, 1966)},
  pages 349--355. Gordon and Breach, New York; Dunod, Paris, 1967.

\end{thebibliography}
